\documentclass[12pt]{amsart}
\usepackage{amsthm,amstext,amsmath,amscd,amssymb,latexsym,mathrsfs}

\usepackage{color}
\usepackage{amsfonts,array}
\usepackage{todonotes}
\usepackage{verbatim}
\usepackage{float}	
\usepackage{pgf,tikz}
\usepackage[toc,page]{appendix}
\usepackage[all,cmtip]{xy}
\usepackage{stmaryrd}
\usepackage{xparse}
\usepackage{listings}
\usepackage{enumitem}
\setlist{nosep, left=0pt}

\usepackage{tabularx,booktabs}

\usepackage{ragged2e} 
\usepackage{needspace}   
\usepackage{ltablex}   
\keepXColumns         
\usepackage{placeins} 

\newcolumntype{L}{>{\RaggedRight\arraybackslash}p{3.5cm}}
\newcolumntype{Y}{>{\RaggedRight\arraybackslash}X}

\usepackage{caption}
\usepackage{subcaption}
\usepackage[T1]{fontenc}
\usepackage{lmodern}
\usepackage{etoolbox}
\usepackage{xcolor}

\usepackage[shortalphabetic]{amsrefs}
\usepackage[all]{xy}
\usepackage{newtxtext}
\usepackage{newtxmath}
\usepackage{framed}
\usepackage{graphicx}
\usepackage{wasysym}
\usepackage{pdfpages}

\usepackage[thinlines,thiklines]{easybmat}

\usepackage{microtype} 

\usepackage{mathtools}  

\usepackage{bbm}
\usepackage[bookmarks=true, colorlinks=true, linkcolor=blue!50!black,
citecolor=GoetheBlue, urlcolor=orange!50!black, pdfencoding=unicode]{hyperref}
\usepackage[left=3.5cm,right=3.5cm,top=3cm,bottom=3cm,footskip=1cm
]{geometry}
\makeatletter
\def\thm@space@setup{%
  \thm@preskip=5pt \thm@postskip=5pt
}
\makeatother

\makeatletter
\renewcommand\@biblabel[1]{[#1]\quad}%

\@ifpackageloaded{amsrefs}{%
  \ifcsname BibLabel\endcsname
    \usepackage{etoolbox}
    \patchcmd{\BibLabel}{\hfill}{}{}{}%
    \renewcommand{\BibLabel}{%
      \ifcsname Hy@raisedlink\endcsname
        \Hy@raisedlink{\hyper@anchorstart{cite.\CurrentBib}\hyper@anchorend}%
        [\thebib]\quad
      \else
        [\thebib]\quad
      \fi
    }%
  \fi
}{}%
\makeatother

\newcommand{\restrict}{\,\rule[-5pt]{0.4pt}{12pt}\,{}}

\numberwithin{equation}{section}

\newcommand{\equ}{\ensuremath{\,=\,}}

\newcommand{\deq}{\ensuremath{\stackrel{\textrm{def}}{=}}}

\DeclareMathOperator{\length}{length}

\newcommand{\Ra}{`\ensuremath{\Rightarrow}'  }
\newcommand{\La}{`\ensuremath{\Leftarrow}'  }
\newcommand{\LI}{`\ensuremath{\subseteq}'  }
\newcommand{\RI}{`\ensuremath{\supseteq}'  }

\newcommand{\BC}{{\mathbb{C}}}

\newcommand{\BN}{{\mathbb{N}}}

\newcommand{\BP}{{\mathbb{P}}}
\newcommand{\BQ}{{\mathbb{Q}}}
\newcommand{\BR}{{\mathbb{R}}}

\newcommand{\BZ}{{\mathbb{Z}}}

\newcommand{\Fa}{{\mathfrak{a}}}

\newcommand{\Fp}{{\mathfrak{p}}}
\newcommand{\Fq}{{\mathfrak{q}}}

\newcommand{\CF}{{\mathcal F}}
\newcommand{\CG}{{\mathcal G}}
\newcommand{\CH}{{\mathcal H}}

\newcommand{\CL}{{\mathcal L}}
\newcommand{\CM}{{\mathcal M}}

\newcommand{\CO}{{\mathcal O}}

\newcommand{\CQ}{{\mathcal Q}}

\DeclareMathOperator{\Spec}{Spec}
\DeclareMathOperator{\Proj}{Proj}
\newcommand{\DSpec}{D\!-\!\Spec}

\DeclareMathOperator{\supp}{supp}

\DeclareMathOperator{\rank}{rank}

\DeclareMathOperator{\Pic}{Pic}
\DeclareMathOperator{\Cl}{Cl}

\DeclareMathOperator{\rad}{rad}

\DeclareMathOperator{\Cox}{Cox}

\DeclareMathOperator{\Rel}{Rel}

\newcommand{\Quot}{\mathop{\rm Quot}\nolimits}

\newcommand{\Gen}{{\rm Gen}}

\definecolor{GoetheBlue}{RGB}{0,97,143}
\usepackage[bookmarks=true, colorlinks=true, linkcolor=blue!50!black,
citecolor=GoetheBlue, urlcolor=orange!50!black, pdfencoding=unicode]{hyperref}
\usepackage[left=3.5cm,right=3.5cm,top=3cm,bottom=3cm,footskip=1cm
]{geometry}
\makeatletter
\def\thm@space@setup{%
  \thm@preskip=5pt \thm@postskip=5pt
}
\makeatother

\newtheorem*{theorem*}{Theorem}

\newtheorem{theorem}{Theorem}[section]
\newtheorem{lemma}[theorem]{Lemma}

\newtheorem{proposition}[theorem]{Proposition}
\newtheorem{corollary}[theorem]{Corollary} 

\theoremstyle{definition}
\newtheorem{definition}[theorem]{Definition}
\newtheorem{example}[theorem]{Example}

\newtheorem{remark}[theorem]{Remark}

\newtheorem{introthm}{Theorem}

\def\phi{\varphi}
\def\epsilon{\varepsilon}
\def\setminus{\smallsetminus}
\let\oldbullet\bullet
\def\bullet{{\mathchoice{\oldbullet}%
                        {\oldbullet}%
                        {\scriptscriptstyle\oldbullet}%
                        {\oldbullet}}}

\let\emptyset\varnothing

\title{Topological and scheme-theoretic properties of the $D$-graded Proj construction} 

\author{Felix G\"obler}

\AtBeginDocument{\def\MR#1{}} 

\begin{document}





\begin{abstract}
We generalize the topological description of the $\BN$-graded Proj construction to the multigraded Proj construction for factorially graded rings that are graded by finitely generated abelian groups $D$.
However, there is one big structural difference: While the classical description is given by the space of homogeneous prime ideals not containing the irrelevant ideal, we characterize the multigraded Proj setting using $D$-prime ideals, i.e.\ ideals that have the prime property, but only for homogeneous factorizations. In particular, we establish a multigraded version of the Nullstellensatz. \newline
Additionally, we present algebraic conditions for separability in terms of factorially graded rings, and observe that $\Proj^D(S)$ is not separated in many cases.

Finally, building on Mayeux--Riche's definition of Serre twists, we give a criterion for their freeness.
\end{abstract}
\maketitle

\tableofcontents


\section*{Introduction}
Projective space is often introduced as the set of `directions' in affine space. Points that differ only by a scalar are identified, and the familiar patchwork of affine charts is glued together along overlaps. 
That construction, however, is only one way to see the same object. Projective space also sits naturally as a topological subspace of the prime spectrum of a graded coordinate ring, and this topological viewpoint highlights how the global shape of a projective variety is encoded in the way its affine pieces meet.

When the grading becomes multi–dimensional (coming, for example, from divisor classes), the combinatorics of degrees enriches that glueing: affine charts can have different dimensions, overlaps can behave unexpectedly, and new phenomena such as persistent non-separatedness appear. 

In particular, the notion of `graded prime' no longer captures the correct geometry, as many graded primes either lie inside the irrelevant part or produce affine patches of the wrong dimension. Thus, the familiar correspondence between degree-zero localizations at homogeneous elements $f$ and their non-vanishing loci $D_+(f)$ in the graded Zariski topology can fail. This paper explains why those failures occur and settles a notion on graded ideals that recovers the classical correspondence.

Let us recall the construction of the scheme structure on the $\BN$-graded $\Proj$ construction via topological spaces (for example, \cite{Bosch} section 9.1). Starting point is a $\BN$-graded ring $S = \bigoplus_{n\ge 0} S_n$ with \emph{irrelevant ideal} $S_+ = \bigoplus_{n > 0} S_n$. The \emph{homogeneous prime spectrum} of $S$ is then defined as
\begin{align*}
    \Proj(S) \deq \{\Fp \in \Spec(S) \mid \Fp \text{ graded}, S_+ \subsetneq \Fp\}.
\end{align*}
Also, let $D_+(f) = \{\Fp \in \Proj(S) \mid f \not\in\Fp\}$ for homogeneous $f \in S$.
The idea is to equip $\Proj(S)$ with the scheme structure obtained from the local homeomorphisms
\begin{align*}
    \psi_f \colon D_+(f) \to \Spec(S_{(f)}), \Fp \mapsto \Fp S_f \cap S_{(f)}.
\end{align*}
Now, the description of $\psi_f$ and the corresponding interpretation of $\Proj^D(S)$ in terms of graded prime ideals have not yet been studied for general multigraded rings.
Therefore, we introduce the notion of $D$-prime ideals and prove that $\psi_f$ is also a homeomorphism in the $D$-graded Proj setting, if we assume $S$ to be factorially $D$-graded.

A homogeneous element $f \in S$ is called \emph{relevant} if the group of units $D^f$ of $S_f$ has finite index in $D$, and the \emph{irrelevant ideal} $S_+$ is defined to be the ideal generated by all relevant elements in $S$.
For flexibility, we might consider subsets of the `irrelevant ideal'. 
Therefore, we call a pair $(S, B)$ a \emph{conical ring} if $S$ is an arbitrary multigraded ring and $B \unlhd S$ is a subset $B \subseteq S_+$.

Then the \emph{$D$-graded Proj of $(S, B)$} is defined as 
\begin{align*}
    \Proj^D_B(S) \deq \bigcup_{f \in B \text{ relevant}} \Spec(S_{(f)}).
\end{align*}

In this setting, we proved the following statement.

\begin{introthm}[\autoref{thm:Proj_D=Proj^D}]
    Let $S$ be a factorially and effectively $D$-graded integral domain and $(S, B)$ a conical ring. Then we can equip its multihomogeneous $D$-prime spectrum 
    \begin{align*}
    \Proj_D^B(S) \deq \{\Fp \unlhd \mid \Fp \text{ is $D$-prime and } B\not\subset \Fp\}
\end{align*}
    with the structure of a scheme in such a way, that for relevant $f \in B$ the homeomorphisms $\psi_f\colon D_+(f) \to \Spec(S_{(f)})$ of topological spaces become isomorphisms of schemes and thus give rise to an affine open covering of $\Proj^B_D(S)$ by the sets $D_+(f)$. The resulting scheme is $\Proj^D_B(S)$, so we have an isomorphism $\Proj^D_B(S) \cong \Proj^B_D(S)$. Therefore, we also call $\Proj^D_B(S)$ the $\Proj$ scheme associated to $(S, B)$.
\end{introthm}

Regarding the separatedness of the Proj construction, there is only a single criterion for detecting large open separated subsets known so far (see \cite{BS}, Proposition 3.3).
We establish new separation criteria in terms of relevant elements (cf.\  Theorem~\ref{thm:lin_dep_not_sep}).
In fact, the only necessary property for $\Proj^D(S)$ to be separated, is that all multiplication maps $\mu_{(fg)}\colon S_{(f)} \otimes S_{(g)} \to S_{(fg)}$ for relevant $f, g \in S$ are surjective.
We translate the surjectivity of $\mu_{(fg)}$ into an algebraic condition for relevant $f$ and $g$ (see Definition~\ref{def:weak} and Proposition~\ref{prop:surj_cond}).

Ultimately, we show that the full Proj construction is not separated in most cases. The reason is that one can typically find a pair $(f, g)$ satisfying the abovementioned algebraic condition, if the grading is torsion-free and $\rank(D) > 1$. 
The existence of such pairs is closely tied to the equations in $D$ that arise from the $D$-graded ring $S$. 
The non-separatedness is mostly induced by the new phenomena in the $D$-graded setting that homogeneous elements can have linearly independent degrees.

\begin{introthm}[\autoref{thm:lin_dep_not_sep}]
Let $S$ be a factorially and effectively $D$-graded integral domain. Then the following holds true:
\begin{enumerate}
    \item We say that there are only linear dependencies of length $1$, if every linear dependency in $D_\BR$ is of type $\deg(g^k) = \deg(f^l)$ for homogeneous $f, g\in S$. In this case, $\Proj^D(S)$ is separated.

    \item We say that there exists at least one non-trivial irreducible dependency if there is at least one linear dependency that is not of length $1$. Then $\Proj^D(S)$ is not separated.
\end{enumerate}
\end{introthm}

In the $\BN$-graded case, morphisms $f\colon T \to \BP^n$ are given in terms of isomorphism classes of pairs $(\CL, \varphi)$, where
\begin{align*}
  \varphi\colon \CO_T^{n+1} \to \CL
\end{align*}
is an epimorphism, $T$ is any scheme and $\CL$ an ample line bundle. However, it seems that there is no generalization to that in the general multigraded case.
The natural generalization would be studying twisting sheaves of $\Proj^D(S)$ for polynomial rings $S$ and ample families, collections $\CL_1, \ldots, \CL_r$ of invertible sheaves, such that $\CL^d = \CL_1^{d_1} \otimes \ldots\otimes \CL_r^{d_r}$, $d \in \BZ^r$, satisfies certain properties, generalizing the properties of ample line bundles.

There is no natural ample family that recovers the $D$-graded Proj in general, because constructing one requires a choice of a basis of $\Pic(X)$, which is not canonical.
Concretely, this is since $\Cl(X)$ may contain torsion or has rank greater than one.

Finally, we were able to give a criterion for the freeness of twisting sheaves $\CO_{\Proj^D(S)}(d)$ in terms of the finite index subgroups $D^f$ arising from relevant $f \in S$.

Concretely, let $S$ be an effectively $D$-graded integral domain, $d\in D$ and $(f_i)_{i\in I}$ a system of generators for the irrelevant ideal $S_+$. Then $\CO_X(d)$ is free if and only if $d \in \bigcap_{i \in I} D^{f_i}$ (see Corollary~\ref{thm:O(d)_free}).


{\normalfont\bfseries Organization of the article.}
In Section~\ref{sec:mult_ring}, we quickly recall the definition and some properties of the $D$-graded Proj construction. 
Section~\ref{sec:mult_spectra} deals with the generalization of the classical Proj description in terms of $D$-prime ideals and the homogeneous Zariski topology.
In Section~\ref{sec:sep_crit} we present several conditions implying non-separatedness of the Proj construction.
Finally, in Section~\ref{subsec:sheaves}, we recall the construction of twisting sheaves for multigraded rings in the lines of \cite{May} and prove a condition for freeness.\\


{\normalfont\bfseries Acknowledgements.}
Again, I would like to thank my (former) supervisor, Alex Küronya, for his tremendous support and insightful feedback.
I am also very grateful to Johannes Horn, Andrés Jaramillo Puentes, Kevin Kühn, Jakob Stix, Martin Ulirsch, and Stefano Urbinati for many helpful discussions and valuable comments.

Partially funded by the Deutsche Forschungsgemeinschaft (DFG, German Research Foundation) TRR 326 \textit{Geometry and Arithmetic of Uniformized Structures}, project number 444845124, and by the LOEWE grant \emph{Uniformized Structures in Algebra and Geometry}.













\section{$D$-graded Proj}\label{sec:mult_ring}
We give a quick overview of the $D$-graded Proj construction. It was first described by \cite{BS} in 2000, and further studied by \cite{KU}, \cite{KSU}, and \cite{May}. For more details on the $D$-graded Proj construction, such as quotient presentations (in the sense of GIT) and stability conditions, we refer to our previous work \cite{paper1}, \cite{paper2}.

Overall, we work with multigraded rings, i.e.\ commutative unital rings that are graded by a finitely generated abelian group. All multigraded rings are assumed to be effectively graded, which means that the grading group $D$ is generated by its support, i.e.\ $D = ( \{d\in D\mid S_d \neq 0\})$. This way, the relevant locus of $S$ is uniquely determined (cf.\ \cite{paper1}, Example 1.12).

We aim to generalize the classical topological description of Proj (cf.\ \cite{Bosch}, §9.1) in Chapter~\ref{sec:mult_spectra}. But in general, we will need to impose several assumptions on $S$.

\begin{definition}\label{def:fact_graded}
    Let $D$ be a finitely generated abelian group and let $S$ be a $D$-graded integral domain. We denote the set of homogeneous ring elements in $S$ by $h(S)$.
    \begin{enumerate}[label=(\arabic*)]
        \item A nonzero nonunit $f \in h(S)$ is a \emph{$D$-prime element} if for all $g, h \in h(S)$ satisfying $f \mid gh$, either $f \mid g$ or $f \mid h$.
        \item Consequently, a graded ideal $\Fp \unlhd S$ is called a \emph{$D$-prime} ideal, if for all homogeneous $f, g \in S$ satisfying $fg \in \Fp$ either $f \in \Fp$ or $g \in \Fp$.
        \item We say that $S$ is \emph{factorially $D$-graded} if every nonzero nonunit $f \in h(S)$ is a product of $D$-prime elements.
    \end{enumerate}
\end{definition}

Recall that being prime and $D$-prime is not equivalent for general $D$-gradings. For example, for $S = \BC[x]$ and $D=\BZ/2\BZ$, where $\deg(x) = 1_D$, the ideal $(x^2-1)$ is $D$-prime but not prime (cf.\ \cite{paper1}, Example 1.7).

The most important notion is that of relevant elements in multigraded rings.
\begin{definition}
    A multigraded ring $S$ is called \emph{periodic} if the degrees of the homogeneous units $f \in S^\times \cap h(S)$ form a finite index subgroup of the grading group.
An element $f \in S$ is called \emph{relevant} if
    \begin{enumerate}[label=(\roman*)]
        \item $f$ is homogeneous and
        \item the localization $S_f$ is \emph{periodic}.
    \end{enumerate}
 We denote the set of relevant elements of $S$ by $\Rel^D(S)$.
For homogeneous $f \in h(S)$, we define the \emph{support group} of homogeneous units in $S_f$ to be
\begin{align*}
    D^f \deq \{\deg_D(g) \mid g \in (S_f)^\times \cap h(S_f)\}.
\end{align*}
Finally, we define the \emph{irrelevant ideal $S_+$} of $S$ to be the ideal generated by all relevant elements of $S$, i.e.\ $S_+ = (\Rel^D(S))$. 
\end{definition}

If $S$ is factorially graded and noetherian, we can choose a minimal generating system of $S$ by $D$-prime elements and hence construct minimal generators for $S_+$ via monomials in those $D$-prime elements. We denote a corresponding minimal generating set for $S_+$ by $\Gen^D(S)$. If $S$ is a polynomial ring, we say that $f \in S$ homogeneous is \emph{monomic relevant} if $f$ is a product of exactly $r$ distinct variables of $S$, such that $D^f$ has finite index in $D$, $f$ is square-free, and its coefficient is $1_S$. In this case, $\Gen^D(S)$ is given in terms of monomic relevant elements (cf. \cite{paper1}, Lemma 1.20).

There are many useful criteria for relevance (see \cite{paper1}, Lemma 1.15). However, we make heavy use of the following conclusion when proving surjectivity of $\psi_f$. The proof can be found in \cite{paper1}, Corollary 1.17.

\begin{corollary}\label{cor:rel_many_el_deg_zero}
    Let $h \in h(S)$ be homogeneous and $f \in S$ relevant. Then there exists an element $g_h \in h(S)$, an element $k \in \BZ$ and an element $N > 0$ such that
    \begin{align*}
        \deg\left(\frac{h^N g_h}{f^k}\right) \equ 0 .
    \end{align*}
\end{corollary}

Now $\Proj^D(S)$ is defined to be the space obtained from gluing all the degree-zero localizations of the correct dimension (cf.\ \cite{paper1}, Lemma 2.18) inside a quotient space $Q(S)$ that exists in the category of ringed spaces. We showed that, on basic affine opens, the structure sheaf of $\Proj^D(S)$ coincides with the restriction of the structure sheaf of $Q(S)$. This is the very reason why we can glue the collection of all $\Spec(S_{(f)})$ for relevant $f\in S$ `inside' $Q(S)$, giving rise to the structure of a scheme. More details on why this embedding works can be found in \cite{paper1}, Section 1.3.

\begin{definition}[Multigraded spectrum of a multigraded ring]\label{def:multihom_spec}\makebox{}{}\\
The \emph{multigraded spectrum of $S$} is defined as the scheme
\begin{align*}
    \Proj^D(S) \deq \bigcup_{f \in \Rel^D(S)} \Spec(S_{(f)}) \subseteq \CQ(S) .
\end{align*}
\end{definition}

\begin{example}\label{ex:torsion_2Z_1st}
   Let $D= \BZ/2\BZ$ and $S = \BC[x]$ graded by $\deg(x) = 1_D$. Then, as $D$ is finite, any homogeneous element in $S$ is relevant. In particular, $1 \in S$ is relevant and $\Proj^D(S) = \Spec(S_0) = \Spec(\BC[x^2])$.
\end{example}


\section{Multihomogeneous \texorpdfstring{$D$}{D}-prime Spectra}\label{sec:mult_spectra}
Let $D$ be a finitely generated abelian group of rank $r$ and $S$ be an effectively $D$-graded commutative ring.
Generalizing the construction of projective space, we attached a scheme, namely the Brenner--Schröer Proj, to a multigraded ring, so that it is compatible with the GIT quotient obtained from the natural action of $G = \Spec(S_0[D])$ on $\Spec(S)\setminus V(S_+)$ (cf. \cite{paper1}, Proposition 2.13).

In this chapter, we want to give the generalization of the Proj construction via topology along the lines of \cite{Bosch}, §9.1, to factorially $D$-graded integral domains.
We start with an example that indicates that $f$ being relevant could be equivalent to 
\begin{align*}
    \psi_f\colon D_+(f) \to \Spec(S_{(f)}),\ \Fp \mapsto \Fp S_f \cap S_{(f)}
\end{align*}
being a homeomorphism.  

\begin{example}\label{ex:not_rel_no_homeo}
Let $S = \BC[x, y, z]$, with grading $x\mapsto e_1$, $y\mapsto e_2$, $z \mapsto e_1+e_2$.
    It holds that
    \begin{align*}
        S_{(xy)} \equ \BC[\frac{z}{xy}] , \ 
        S_{(xz)} \equ \BC[\frac{xy}{z}] \equ S_{(yz)}, 
    \end{align*}
    and all intersections coincide:
    \begin{align*}
        S_{(x^2yz)} \equ S_{(xy^2z)} \equ S_{(xyz^2)} \equ \BC[\frac{xy}{z}, \frac{z}{xy}].
    \end{align*}
    On the other side, the grading is torsion-free, so that the graded prime ideals not containing $S_+ = (xy, xz, yz)$ are given by
    \begin{align*}
       \Proj_D(S) \deq \{(x), (y), (z), (0), (\alpha z+ xy) \mid \alpha \in \BC^\times\}.
    \end{align*}
    Note that $(x,y), (x,z)$ and $(y, z)$ each contains $S_+$, so they do not contribute to $\Proj_D(S)$.
With notation as in \cite{Bosch}, Proposition 9.1/12 (i.e. $D_+(f) = \{\Fp \in \Spec(S) \mid S_+ \not\subset \Fp \text{ and } \Fp \text{ is a graded prime ideal}\}$) we compute
\begin{align*}
    D_+(xy) &\equ \{ (0), (z), (xy-z), (z-xy), \ldots\} \\
    D_+(xz) &\equ \{ (0), (y), (xy-z), (z-xy), \ldots\} \\
    D_+(yz) &\equ \{ (0), (x), (xy-z), (z-xy), \ldots\} \\
    D_+(xyz) &\equ \{(0), (xy-z), (z-xy), \ldots\},
\end{align*}
so $\psi_f$ is a homeomorphism for $f$ monomic relevant. 
Note that $S_{(xz)} = S_{(yz)} = S_{(z)}$, so we might say $z$ is relevant. But
\begin{align*}
    D_+(z) \equ \{(0), (x), (y), (xy-z), (z-xy), \ldots\}
\end{align*}
and therefore $\psi_z(x) = \psi_z(y) = (\frac{xy}{z}) \in \Spec(S_{(z)}) \subseteq \Proj^D(S)$. In particular, $\psi_z$ is not a homeomorphism.
\end{example}

In general, $\psi_f$ need not be a homeomorphism in the Zariski topology. This phenomenon was already noted in \cite{BS}, Remark 2.3. We now provide an example illustrating this:

\begin{example}\label{ex:D_prime_homeo}
    Let $S = \BC[x]$ and $D=\BZ/2\BZ$, where $\deg(x) = 1_D$. As $1$ is relevant, we want $\psi_1 \colon D_+(1) \to \Spec(S_{(1)}) = \Spec(S_0)$ to be a homomorphism. But the graded prime ideals of $S$ are given by $(0)$ and $(x)$. Hence $x^2\pm 1 \in \Spec(S_0)$ does not have a preimage. It is immediate that $\psi_1$ becomes a homeomorphism if we define $D_+(1)$ to be the set of all $D$-prime ideals of $S$, as $x^2\pm 1 \in D_+(1)$ in this case.
\end{example}


\subsection{Homogeneous Zariski Topology}\label{sec:D_Spec}

Motivated by Example~\ref{ex:D_prime_homeo}, we define a new topology that allows us to generalize \cite{Bosch}, Proposition 9.1/13 (also see \cite{EGA2}, Proposition (2.3.6)).
To the best of our knowledge, this point of view has not been addressed in the literature yet.

\begin{definition}\label{def:D_prime_spec}
    Let $S = \bigoplus_{d\in D}S_d$ be a multigraded ring. Then
    \begin{align*}
        \DSpec(S) \deq \{\Fp \unlhd S \mid \Fp \text{ is $D$-prime}\}
    \end{align*}
    is called the \emph{$D$-prime spectrum} of $S$ and
    \begin{align*}
        \Proj_D(S) \deq \{\Fp \in \DSpec(S) \mid S_+ \not\subset \Fp\}
    \end{align*}
    is called the \emph{multihomogeneous $D$-prime spectrum} of $S$.
    For a homogeneous subset $E \subseteq S$ (or the graded ideal $\Fa \unlhd S$ generated by $E$) we define
\begin{align*}
    V_+(E) \deq \{\Fp \in \Proj_D(S) \mid E \subseteq \Fp\}.
\end{align*}
\end{definition}

\begin{remark}\label{lem:DSpec_functorial}
The $D$-prime spectrum still satisfies the necessary functorial properties: 

Let $\varphi \colon R\to S$ be a morphism of multigraded rings and $\Fp \unlhd S$ a $D_S$-prime ideal. Then $\varphi^{-1}(\Fp) \unlhd R$ is $D_R$-prime.
\end{remark}


\begin{proposition}\label{prop:hom_top}
    By defining the sets of type $V_+(E)$ for $E \subseteq S$ homogeneous (and radical, if $D$ has torsion) as closed subsets of $X = \Proj_D(S)$, we get a topology on $X$, called the \emph{homogeneous Zariski topology}. Concretely, the following properties hold true for homogeneous subsets $E, E' \subseteq S$ and families of homogeneous subsets $(E_\lambda)_{\lambda \in \Lambda}$ respectively:
    \begin{enumerate}[label=(\roman*)]
        \item $V_+(0) \equ \Proj_D(S)$, $V(1) \equ \emptyset$.
    
        \item If $E \subseteq E'$, then $V_+(E') \subseteq V_+(E)$. 
        \item $V_+(\bigcup_{\lambda \in \Lambda} E_\lambda) \equ V_+(\sum_{\lambda \in \Lambda} E_\lambda) \equ \bigcap_{\lambda \in \Lambda} V_+(E_\lambda)$.
        
        \item $V_+(EE') \equ V_+(E) \cup V_+(E')$.
        \item $V_+(E) \equ V_+(\rad(E))$.
    \end{enumerate}
    In particular, the sets of type 
    \begin{align*}
        D_+(f) \deq \{\Fp \in \Proj_D(S) \mid f \not\in\Fp\}
    \end{align*}
    for relevant $f \in S$ form a basis of the topology on $X$. 
\end{proposition}

\begin{proof}
    Assertions (i) - (iii) follow from \cite{Bosch}, Proposition 6.1/1. For (iv), let $\Fp \in \Proj_D(S)$ such that $\Fp \not\in V_+(E), V_+(E')$, i.e.\ $E, E' \not\subset \Fp$. Then there exist homogeneous elements $f \in E, f' \in E'$ such that $ff'\not\in \Fp$. As $\Fp$ is prime on homogeneous elements, it follows $f, f' \not \in \Fp$, thus $\Fp \not\in V_+(EE')$. Regarding $(v)$, we have to consider two cases. First, if $D$ has torsion, we demand that $E$ is radical, so $(v)$ holds for trivial reasons. If $D$ is torsion-free, then $D$ has a total order. In particular, $\rad(E)$ is a homogeneous subset of $S$ and $D$-prime ideals are prime. Thus for $\Fp \in D(S_+)$ we have $E \subseteq \Fp$ if and only if $\rad(E) \subseteq \Fp$. \\

For the last claim, let $U \subset \Proj_D(S)$ be open and $\Fa \unlhd S$ be graded such that $\Proj_D(S) \setminus U = V_+(\Fa)$. Thus
    \begin{align*}
        V_+(\Fa) \equ \bigcap_{f \in \Fa \cap \Rel^D(S)} V_+(f),
    \end{align*}
    which induces
    \begin{align*}
        U \equ \Proj_D(S) \setminus V_+(\Fa) \equ \bigcup_{f \in \Fa \cap \Rel^D(S)} D_+(f) .
    \end{align*}
\end{proof}

We can formulate a notion of relevance in this context:

\begin{definition}
    Let $E \subseteq S$ be a homogeneous subset. We define the \emph{support of $E$ (resp. of $\Fa = (E)$)} to be the set of all degrees contained in the ideal generated by $E$, i.e.
    \begin{align*}
        \supp(E) \deq \{\deg(g) \mid g \in (E)\} .
    \end{align*}
\end{definition}

\begin{corollary} 
Let $S$ be a factorially $D$-graded noetherian domain, and $r = \rank(D)$.
    An element $f \in S$ is relevant if and only if $V_+(f)$ has at least $r$ irreducible components such that their support generates a group of rank $r$.
\end{corollary}

\begin{proof}
    \Ra   Let $f$ be relevant. We assume without loss of generality that $f$ has exactly $r$ irreducible homogeneous divisors $g_1, \ldots, g_r$ (as $S$ is noetherian and factorially graded). In particular, there are $r$ minimal $D$-prime ideals $\Fp_i = (g_i)$ over $f$, giving us 
    \begin{align*}
        V_+(f) \equ (D_+(f))^c \equ \left(\bigcap_{i=1}^r D_+(g_i)\right)^c \equ \bigcup_{i=1}^r V_+(g_i).
    \end{align*}
    Note that the sets $V_+(g_i)$ are irreducible as the $g_i$ are. Since $f$ is relevant, the support of $f$, given by the union of $\supp(\Fp_i)$, generates a group of rank $r$.  

    \La Conversely, we may assume that $V_+(f)$ has exactly $r$ irreducible components without loss of generality, i.e.\ $V_+(f) = \cup_{i=1}^r X_i$ for some irreducible subsets $X_i \subseteq V_+(f)$. In particular, $X_i = V_+(\Fp_i)$ for some $D$-prime ideal minimal over $f$, so that the $\Fp_i$ correspond to the irreducible divsors of $f$. Thus, we have $r$ distinct irreducible homogeneous divisors of $f$, such that their degrees span a cone of rank $r$, hence $f$ is relevant.
\end{proof}

The following statement is a generalization of \cite{Bosch}, 9.1/12. It was already stated in \cite{KU}, Lemma 3.5, but without a proof.

\begin{lemma}\label{lem:basis_shifted}
Let $S$ be a $D$-graded ring.
For $n \in \BN_+$ we define the graded ring $S^{(n)}$ to be $(S^{(n)})_d := S_{nd}$. Then
\begin{align*}
    \Proj^D(S) \stackrel{\sim}{\equ} \Proj^D(S^{(n)}).
\end{align*}
In particular, for $d \in D$ the sets of type $D_+(f)$ for relevant $f \in S_{nd}$, such that $n \in \BN_+$ varies, form a basis of the homogeneous Zariski topology on $\Proj_D(S)$.
\end{lemma}

\begin{proof}
It is sufficient to show that both rings ($S$ with the two different gradings) share the same relevant locus. But $f \in S$ is $D$-relevant if and only if the degrees $\deg_D(g)$ generate a finite index subgroup of $D$, where $g$ is a homogeneous divisor of $f^m$ for some positive integer $m$. Let $g \in S_d$ be such a divisor, then clearly $g^n$ is still a homogeneous divisor of some power of $f$. In other words, $g^n \in (S^{(n)})_d$ and $f$ is also $D^{(n)}$-relevant (i.e.\ $d \in D^f$ if and only if $d \in (D^{(n)})^f$).

Conversely, let $f \in S^{(n)}$ be $D^{(n)}$-relevant and $g \in (S^{(n)})_d$ a homogeneous divisor of $f^m$, thus $dn \in (D^{(n)})^f$. But $g$ has $D$-degree $d$, so clearly $d \in D^f$ (and therefore $dn \in H_f$ since $n \in \BN_+$) and $f$ has to be also $D$-relevant.
\end{proof}

Most parts of the Proj construction are sensitive to the chosen irrelevant ideal. As we have already seen in examples, the whole Proj might be non-separated. Thus, we want to consider subsets of $S_+$. Let $B \unlhd S$ be a subset of $S_+$. Then the pair $(S, B)$ will be called a \emph{conical ring}.

\begin{definition}\label{def:hom_top_B}
Let $(S, B)$ be a conical ring. We call
\begin{align*}
    \Proj_D^B(S) \deq \{\Fp \in \DSpec(S) \mid B\not\subset \Fp\}
\end{align*}
the \emph{multihomogeneous $D$-prime spectrum} of $(S, B)$ (w.r.t $D$).
Therefore we call $V_+(B) = \{\Fp \mid \Fp \text{ is $D$-prime and } B \subseteq \Fp\}$ the \emph{irrelevant subscheme} and $D_+(B) := \Proj_D^B(S) \setminus V_+(B)$ the \emph{relevant locus} of $\Proj_D^B(S)$. For $f \in \Rel^D_B(S)$ we define
\begin{align*}
    D_B(f) \deq \{\Fp \in \Proj_D^B(S) \mid f \not\in \Fp\}.
\end{align*}
\end{definition}

\begin{example}
    Consider $S = \BC[x]$ and $D=\BZ/2\BZ$ from Example~\ref{ex:D_prime_homeo}. It holds $\deg(x) = 1$ and hence
    \begin{align*}
        \Proj^D(S) \equ \Proj_D(S) \equ \DSpec(S),
    \end{align*}
    where the last equation is due to $S_+ = (1)$ and therefore $V(S_+) = \emptyset$. In particular, 
    \begin{align*}
        \DSpec(S) \equ \Spec(S_0).
    \end{align*}
\end{example}

We use the following (technical) language to grasp the structure of $\psi_f$.

\begin{definition}\label{def:el_deg_zero}
Let $(S, B)$ be a factorially graded conical ring.
    Let $f\in \Rel^D_B(S)$ and $0 \neq g \in S_e$.
    Furthermore, we assume that $1_S$ is not relevant.
    As $S$ is factorially graded, we can bring rational functions $\frac{g}{f}$ in reduced fractional form and talk about homogeneous elements being coprime (i.e.\ having no common non-unit homogeneous factor, where units are taken with respect to $S_0$).
    \begin{enumerate}[label=(\arabic*)]
        \item Let $r, s \in S$ be homogeneous coprime elements. We say that $\frac{r}{s}$ is a  \emph{non-trivial element of degree zero in $\Quot(S)$} if there exists some relevant $f' \in S$ such that $\frac{r}{s} \in S_{(f')}\setminus S_0$. If the specific relevant element is given, we might also say that $\frac{r}{s}$ is a \emph{non-trivial element of $S_{(f')}$}.

         \item Let $f' \in S$ be relevant. We say that a non-trivial element $\frac{r}{s}$ of $S_{(f')}$ is a \emph{generator of $S_{(f')}$} as $S_0$-algebra, if $\frac{r}{s}$ is part of a minimal generating system of $S_{(f')}$ as $S_0$-algebra.
         We say that a non-trivial element of degree zero $\frac{r}{s}$ in $\Quot(S)$ is a \emph{generator of degree zero}, if there exists some relevant $f'' \in S$ such that $\frac{r}{s}$ is a generator of $S_{(f'')}$.
            
        \item We say that $\frac{g^k}{f^l}$ \emph{represents a non-trivial element of degree zero in $S_{(f)}$}, if there is a non-trivial element $\frac{r}{s}$ of $S_{(f)}$ and integers $k, l\in \BN$, such that the reduced fractional form of $\frac{g^k}{f^l}$ is equal to $\frac{r}{s}$. 
        
        \item Let $g'\mid f$ be some homogeneous divisor of $f$.
        We say that the pair $(g, g')$ \emph{admits a non-trivial element of $S_{(f)}$}, if there exists a homogeneous element $h \in S\setminus\{0\}$ and integers $k, l \in \BN$ such that $\frac{g^k h}{(g')^l}$ is a non-trivial element of $S_{(f)}$, i.e.\ $\deg(g^k h) - \deg((g')^l) = 0$ and $\frac{g^k h}{(g')^l}\not\in S_0$. As the exponents of $g'$ in $S_{(f)}$ are natural numbers, clearly there exists a minimal exponent $n_h$ of $g'$ such that there exists an element $0\neq h \in h(S)$ and $k \in \BN$ such that $\frac{g^kh}{(g')^{n_h}}$ is a non-trivial element of $S_{(f)}$, i.e.\ $n_h$ is minimal with the property that $\frac{g^kh}{(g')^{n_h}}$ is a non-trivial element of $S_{(f)}$.

        \item Let $g , g' \in S$ be homogeneous. In general we say that the pair $(g, g')$ \emph{admits a non-trivial element of degree zero (in $\Quot(S)$)}, if there exists some relevant $f' \in S$ such that $(g, g')$ admits a non-trivial element of $S_{(f')}$.
    \end{enumerate}    
\end{definition}

\begin{remark}
\begin{enumerate}
    \item We may replace $g'$ by $f$ in Definition~\ref{def:el_deg_zero}, as we can get rid of the extra factors $\frac{f}{g'}$ by just multiplying them with $h$. We decided on this formulation so that (5) reads clearly.

    \item In Situation (4), note that $h$ may not be unique, only $\deg(h)$ is, and thus the exponent of $f$. 
\end{enumerate}
\end{remark}

    The pair $(g, f)$ admits a non-trivial element of $S_{(f)}$ if and only if $(gS_f) \cap S_{(f)}$ contains an element that does not belong to $S_0$. Note that this is exactly the origin of the element $h$ in Definition~\ref{def:el_deg_zero} (4), i.e.\ we multiply $g^k$ with $\frac{h}{f^l} \in S_f$ and as this is also supposed to be an element of $S_{(f)}$, it holds $\deg(g^kh) = \deg(f^l)$.
    In particular, if $D = \BN$ then $h = 1$, as for two given homogeneous elements $a, b$ with degrees in $\BN$, there always exist some $\alpha, \beta \in \BN$ such that $\deg(a^\alpha) = \deg(b^\beta)$.

\begin{example}
    \begin{enumerate}
        \item Let $S = \BC[x, y, z]$ where $x,y \mapsto e_1$ and $z \mapsto e_2$. As we only have the linear dependency $\deg(x) = \deg(y)$, $\deg(z)$ is linearly independent with respect to the other variables. In particular, $z$ cannot contribute to a non-trivial element of degree zero in $\Quot(S)$. In other words, for all $f \in \Gen^D(S) = \{xz, yz\}$, $(z, f)$ does not admit a non-trivial element of degree zero. 

        \item Let $S = \BC[x, y, z]$ and $x \mapsto e_1, y \mapsto e_2, z \mapsto e_1 + e_2$, $f =xz$ and $g = xy-z$ (cf.\  Example~\ref{ex:not_rel_no_homeo}). 
        Then $(g, f)$ admits a non-trivial element of degree zero in $S_{(f)}$ by taking $h = x$, i.e.\ $\frac{gh}{f}\equ \frac{(xy-z)x}{xz} = \frac{xy-z}{z} = \frac{xy}{z}-1 \in S_{(f)}$. The  generators of degree zero over $S_0$ are given by the elements $\frac{xy}{z}$ for $f= xz, yz$ and $\frac{z}{xy}$ for $f = xy$.
    \end{enumerate} 
\end{example}

The following proposition explains where the element $h$ in the definition of a non-trivial element of $\Quot(S)$ comes from.

\begin{proposition}\label{prop:el_deg_zero}
Let $S$ be a factorially $D$-graded integral domain and let $f\in\Rel^D_B(S)$.
Furthermore, let $\Fp\in D_+(f)$ and let $g\in\Fp_e$ be homogeneous.
If there exist a homogeneous element $h\in S$ and $k,l\in\mathbb N$ with
\[
\deg(g^k h)=\deg(f^l),
\]
then $h\notin\Fp$.
\end{proposition}

\begin{proof}
All we need is a reason why not all factors of $g^k \cdot h$ can be elements of $\Fp$. First of all, we clearly have $g \in \Fp$.
Now since $(g,f)$ admits a non-trivial element of $S_{(f)}$, $(gS_f) \cap S_{(f)}$ contains an element that is not in $S_0$, so either $\deg(g)$ and $\deg(f)$ are linearly dependent (i.e.\ $h=1$ and $\deg(g^k) = \deg(f^l)$ for some $k, l$) or the numerator of $\frac{g^k h}{f^l}$ is an element of $(p)S_f$, so that $h\neq 1$ and $\deg(g^kh) = \deg(f^l)$ for some $h \in S\setminus \Fp$.
In the first case, we may choose $h = 1\not\in\Fp$ and in the second case $h \not\in \Fp$ by construction.
\end{proof}

In Proposition 2.10 (ii), \cite{May} proof the generalization of \cite{Bosch}, Lemma 9.1/7. Concretely, for an $\BN$-graded ring $S$ and homogeneous $f, g \in S$, the localization $S_{(fg)}$ is always isomorphic to $S_{(f)}$ localized at $f^{-\deg(g)} g^{\deg(f)}$.
In the general $D$-graded setting considered by Mayeux--Riche, such a canonical choice is no longer available, and one only obtains an existential description of the isomorphism.
In this paper, we restrict to factorially graded rings and thus can determine when we really do have new elements in the right-hand localization.

\begin{lemma}\label{lem:9.1/7_gen}
    Let $f \in \Rel^D_B(S)$ and $g \in S_e$ be non-zero. Then there is a canonical isomorphism
    \begin{align*}
       \varphi_{f, g}\colon S_{(fg)} \stackrel{\sim}{\to} (S_{(f)})_\CG,
    \end{align*}
   where $\CG$ is the multiplicative subset generated by all generators of degree zero in $\Quot(S)$ with denominator $g'$, where $g' \mid g$ and all generators of degree zero with denominator $h$, such that there exist elements $x \mid f$, $y \mid g$, $xy \not\mid f, g$ but $xy \mid h$.
\end{lemma} 

\begin{proof}
It suffices to show that both rings have the same minimal generating set as $S_0$-algebras. Let $(h, fg)$ admit a non-trivial element of degree zero in $\Quot(S)$, whose reduced fractional form is a generator of degree zero. Then there are three possibilities: (i) the reduced fractional form has either a divisor of $f$, (ii) a divisor of $g$, or (iii) a product of type $xy$ for $x \mid f$, $y \mid g$ but $xy \nmid f, g$ as denominator. Now all elements where the reduced fractional form has a divisor of $f$ as denominator are contained in $S_{(f)}$ and all elements of type (ii) and (iii) are adjoined by $\CG$. Thus, the claim follows.
\end{proof}

\subsection{Topological description of Proj}

The prerequisites are met and we can show that for $f \in \Rel^D_B(S)$ of degree $d \in D$, the map
\begin{align*}
    \psi_f\colon D_+(f) \to \Spec(S_{(f)}),\ \Fp \mapsto \Fp S_f \cap S_{(f)} \, 
\end{align*}
 is a homeomorphism in the factorial setting. We will need this assumption to prove injectivity and surjectivity.

\emph{$\psi_f$ is well defined:} Let $\Fp \in D_+(f)$, $\Fq := \psi_f(\Fp)$ and $g \cdot h \in \Fq$ for $g = \alpha f^{-k}, h= \beta f^{-l} \in S_{(f)}$ (where $\deg(\alpha \beta) = (k+l)d$).
Since $gh \in \Fp S_f\cap S_{(f)}$, $\alpha \beta \in \Fp$. But then we know that $\alpha \in \Fp$ or $\beta \in \Fp$ since they are homogeneous, that is, $g \in \Fq$ or $h \in \Fq$. \\

\emph{$\psi_f$ is continuous:} Let $h \in S_{nd}$ be relevant, $n \in \BN$. Then clearly $h \in \Fp$ if and only if $hf^{-n} \in \psi_f(\Fp)$. Since we know that the sets $\Spec(S_{(h f^{-n})})$ form a basis of the topology on $\Spec(S_{(f)})$ we can deduce
\begin{align}\label{eq:Bosch_12}
    \Fq \in \Spec(S_{\left(\frac{h}{f^n}\right)}) \iff \frac{h}{f^n} \not\in \Fq \iff h \not\in \Fp \iff \Fp \in D_+(h)
\end{align}
and $\psi_f$ is continuous.

\begin{proposition}\label{prop:psi_f_injective} 
Let $S$ be a factorially $D$-graded integral domain.
Then $\psi_f$ is injective.
\end{proposition}

\begin{proof}
Let $\Fp, \Fp' \in D_+(f)$ be graded with $\psi_f(\Fp) = \psi_f(\Fp')$.
First note that $\psi_f$ is injective if and only if no non-zero $D$-prime ideal not containing $f$ is mapped to the zero ideal in $S_{(f)}$.
In particular, given $g \in \Fp \cap S_e$ for $e \in D$ arbitrary, we see that $\psi_f(g) \neq \frac{0}{1}$ if and only if $(g, f)$ admits a non trivial element of degree zero in $S_{(f)})$. 
Thus, if $(g,f)$ does not admit a non-trivial element, the equality $\psi_f(\Fp) = \psi_f(\Fp')$ forces $g$ to be in both or neither ideal.
Hence we apply Proposition~\ref{prop:el_deg_zero}:
\begin{align*}
    &\makebox{}{}\quad g \in \Fp_e \text{ such that } (g, f) \text{ admits a non trivial element of $S_{(f)}$} \\
    &\mathrel{\makebox[7em][c]{\ensuremath{\stackrel{\ref{prop:el_deg_zero},\ref{def:el_deg_zero}(4)}{\iff}}}} 
      \exists h_g \in S\setminus \Fp,\ \exists k, n_h:\ \frac{g^k h_g}{f^{n_h}} \in \psi_f(\Fp) \\
    &\mathrel{\makebox[7em][c]{\ensuremath{\stackrel{\psi_f(\Fp) = \psi_f(\Fp')}{\iff}}}} 
      \exists h_g \in S\setminus \Fp,\ \exists k, n_h:\ \frac{g^k h_g}{f^{n_h}} \in \psi_f(\Fp') \\
    &\mathrel{\makebox[7em][c]{\ensuremath{\stackrel{\ref{prop:el_deg_zero}}{\iff}}}} 
      \exists h_{g'} \in S\setminus \Fp',\ \exists k, n_h:\ \frac{g^k h_{g'}}{f^{n_h}} \in \psi_f(\Fp') \\
    &\mathrel{\makebox[7em][c]{\ensuremath{\stackrel{\ref{def:el_deg_zero}(4)}{\iff}}}} 
      g \in \Fp'_e \text{ such that } (g, f) \text{ admits a non trivial element of $S_{(f)}$.}    
\end{align*}
Now the claim follows by Lemma~\ref{lem:basis_shifted}, (\ref{eq:Bosch_12}) and the fact that $\Fp$ and $\Fp'$ are graded. 
\end{proof}

Regarding the surjectivity, we will make heavy use of Corollary~\ref{cor:rel_many_el_deg_zero}. The proof is rather technical.

\begin{proposition}\label{cor:psi_f_surj}
Let $S$ be a factorially $D$-graded integral domain. Then $\psi_f$ is also surjective.
\end{proposition}

\begin{proof}
We will follow the proof of \cite{Bosch}, Proposition 9.1/13, and generalize it from $\BZ$-gradings to $D$-gradings. For this generalization, we have to construct a preimage $\Fp \in D_+(f)$ of an ideal $\Fq \in \Spec(S_{(f)})$ for some relevant $f \in S_d$, i.e.\ the grading has to be given explicitly. We claim that the grading is given by
 \begin{align*}
     \Fp_e &\deq \{g\in S_e \mid \exists k,l \in \BZ, h \in h(S)\setminus\Fp: \frac{g^kh}{f^l} \in \Fq \}.
 \end{align*}
 In other words, given $f$, the graded part $\Fp_e$ is exactly the set of all $g \in S_e$ that admit a non-trivial element of degree zero in $S_{(f)}$.
\begin{enumerate}[label=(\alph*)]
    \item $\Fp_e \le S_e$:  Let $g, h \in \Fp_e$. Then there are elements $h_1, h_2 \in h(S)\setminus\Fp$, $k, l, m, p \in \BN$ such that
    \begin{align*}
        \frac{g^kh_1}{f^l} \in \Fq \text{ and } \frac{h^mh_2}{f^p} \in \Fq.
    \end{align*}
    For $t = km$, $s = lp$ and $h_3 = h_1^m h_2^k$, we see that
    \begin{align*}
        \frac{(g+h)^t h_3}{f^s} \in \Fq
    \end{align*}
    by the binomial formula. Hence $g+h \in \Fp_e$.

    \item $S_{d_1} \Fp_{d_2} \subseteq \Fp_{d_1+d_2}$: 
    Let $g \in \Fp_{d_2}$ and $a \in S_{d_1}$. For $ag$ to be in $\Fp_{d_1+d_2}$, we need to find $h' \in h(S)\setminus\Fp$ and $m, p \in \BN$ such that $\frac{(ag)^mh'}{f^p}\in \Fq$. As $g \in \Fp_{d_2}$, we know that there already are elements $h \in h(S)\setminus\Fp$ and $k, l \in \BN$ such that $\frac{g^kh}{f^l} \in \Fq$. In addition, $f$ is relevant, so we know that there exists a positive integer $N$ such that $N\deg(a) \in D^f$ by \cite{paper1}, Lemma 1.15 (3). This means that the degree of $a^N$ coincides with the degree of some homogeneous divisor of $f$ (cf.\  Corollary~\ref{cor:rel_many_el_deg_zero}). So there exist elements $h_a \in h(S)$ and $p' \in \BN$ such that $\frac{a^Nh_a}{f^{p'}} \in \Fq$ and the claim follows for $h' = h^N h_a^k$,  $m=Nk$ and $p = p'+l$.

    \item $\Fp$ is $D$-prime: Let $x \in S_{d_1}$ and $y \in S_{d_2}$ such that $xy \in \Fp_{e_1+e_2}$. By definition, there are elements $h \in h(S)\setminus\Fp$ and $k, l \in \BN$ such that $\frac{(xy)^kh}{f^l}\in \Fq$. We want to write this fraction as a product of two fractions, such that one has a numerator divisible by $x$ but not $y$, and the other way around. For this, we use again that some multiples of $\deg(x)$ and $\deg(y)$ have to be elements of $D^f$, say $x^N$ and $y^M$. Thus there are elements $h_x, h_y \in h(S)$ and $r, s \in \BN$ such that $\frac{x^N h_x}{f^r} \in \Fq$ and $\frac{y^M h_y}{f^s}\in \Fq$. As $S$ is factorially graded, we can assume that there exists some integer $\delta$ such that
    \begin{align*}
        \frac{(xy)^kh}{f^l} \equ \left(\frac{x^N h_x}{f^r} \cdot \frac{y^M h_y}{f^s}\right)^\delta.
    \end{align*}
    In particular, one of the factors has to be an element of $\Fq$ as $\Fq$ is prime. Thus $x \in \Fp$ or $y \in \Fp$ and $\Fp$ is $D$-prime.

    \item $f\not\in\Fp$: Works as in \cite{Bosch}. If $f \in \Fp$ then, $1=\frac{f}{f} \in \Fq$, which is not possible.

    \item $\psi_f(\Fp) = \Fq$: 
    \RI For $x = \frac{a}{f^l}\in \Fq$ we see that $a \in \Fp_{\deg(a)}$ by definition. Hence $x \in \Fp S_f \cap S_{(f)}$.
    \LI Take $x=\frac{g^kh}{f^l}\in \Fp S_f\cap S_{(f)}$.
    Then $g \in \Fp$ (cf.\  Proposition~\ref{prop:el_deg_zero}), so there exists some $h_g \in h(S)\setminus\Fp$ and $r, s \in \BN$ such that $\frac{g^r h_g}{f^s} \in \Fq$. As before, we assume there are elements $h'' \in h(S)$, $t \in \BN$, and a positive integer $N$ such that $\frac{h_g^N h''}{f^t}$ has degree zero, as $f$ is relevant. 
    If $h'' \in \Fp$, we may assume that $\frac{h_g^N h''}{f^t}$ coincides with some power of $\frac{g^kh}{f^l}$ by the factoriality and the claim follows. Hence let $h'' \not\in\Fp$. Then
    \begin{align*}
        x^r \cdot \frac{h_g^N h''}{f^t} \in \Fq ,
    \end{align*}
    as we can factor out $\frac{g^r h_g}{f^s}$. Now as $h_g, h'' \not\in \Fp$, $\frac{h_g^N h''}{f^t}$ cannot be an element of $\psi_f(\Fp)$. Thus $x \in \Fq$ by the prime condition.
\end{enumerate}
\end{proof}

\begin{example}\label{ex:double_origin_C_grading}
    Let $S = \BC[x, y, z]$, where the grading is given by $x \mapsto e_1, y \mapsto e_2, z \mapsto e_1 + e_2$. Now take $f =xz$ and $\Fp = (y)$. Clearly, $y \in \Fp_{(0, 1)}$ and $\frac{y \cdot (x^2z)}{xz} = \frac{xy}{z} \in \psi_f(\Fp)$, i.e.\ $h_y = x^2z \not\in\Fp$. We also have for example $yz \in \Fp_{(1, 2)}$ and for $n_h = 2$, $h_{yz} = x^3$ we get $\frac{yz\cdot x^3}{(xz)^2} = \frac{xy}{z} \in \psi_f(\Fp)$.
    If we take $xy^2 \in \Fp_{(1, 2)}$, we get the same parameters, i.e.\ $n_h = 2$ and $h_{xy^2} = x^3$, as $\frac{(xy^2)x^3}{(xz)^2} = (\frac{xy}{z})^2$. In particular, we see that $yz + xy^2 \in \Fp_{(1, 2)}$.
    Also note that $\psi_{xz}(y) = \psi_{yz}(x) (\frac{xy}{z}) \in \Spec(S_{xz}) = \Spec(S_{(yz)})$.
\end{example}
Altogether we can state (cf.\  \cite{Bosch}, Theorem 9.1/17):

\begin{theorem}\label{thm:Proj_D=Proj^D}
    Let $S$ be a factorially $D$-graded integral domain and $(S, B)$ a conical ring. Then we can equip its multihomogeneous $D$-prime spectrum $\Proj_D^B(S)$ with the structure of a scheme in such a way, that for relevant $f \in B$ the homeomorphisms $\psi_f\colon D_+(f) \to \Spec(S_{(f)})$ become isomorphisms of schemes and thus give rise to the affine open covering of $\Proj^B_D(S)$ by the sets $D_B(f)$. The resulting scheme is $\Proj^D_B(S)$, so we have an isomorphism $\Proj^D_B(S) \cong \Proj^B_D(S)$. Therefore we also call $\Proj^D_B(S)$ the $\Proj$ scheme associated to $(S, B)$.
\end{theorem}

\begin{proof}
Directly follows from Proposition~\ref{prop:hom_top}, Proposition~\ref{prop:psi_f_injective} and Proposition~\ref{cor:psi_f_surj}.
\end{proof}

Having proved this structural result, we can deduce the following.

\begin{corollary}\label{cor:intersection_and_open_subsets_of_hom_elements}
Let $S$ be a factorially $D$-graded integral domain, $B \subseteq S_+$, $h \in S$ homogeneous, and $f, g \in \Rel^D_B(S)$ be relevant.
\begin{enumerate}[label=(\alph*)]
    \item If $D$ is finite or $1$ is relevant, then  
        \begin{align*}
        \Proj^D(S) \equ \Spec(S_0) \equ  \Proj_D(S)   \equ \DSpec(S)\ .
    \end{align*}
    \item It holds 
    \begin{align*}
        \Spec(S_{(f)}) \cap \Spec(S_{(g)}) \equ \Spec(S_{(fg)}).
    \end{align*}
In particular, $\Proj^D(S)$ can be seen as the scheme given by the gluing datum given by the $D_+(f)$'s with intersection $D_+(fg)$.
    \item The open subset associated to $h$ is given by 
    \begin{align*}
        D_+(h) \equ \bigcup_{\substack{f \in \Rel^D(S) \\ h \mid f}} D_+(f) \subseteq \Proj^D(S)\ .
    \end{align*}
    If $S$ is also noetherian, we can take the union over all $f \in \Gen^D(S)$ such that $h \mid f$.
\end{enumerate}
\end{corollary}

\begin{proof}
\begin{enumerate}[label=(\alph*)]
    \item As $1$ is relevant, $\psi_1$ is a homeomorphism between $D_+(1) = \Proj_D(S)$ and $\Spec(S_{(1)}) = \Spec(S_0)$, where the latter is equal to $\Proj^D(S)$ by definition.
Since $V_+(S_+) = \emptyset$, the last equation follows.

    \item  As $D_+(f) \cap D_+(g) = D_+(fg)$, the claim follows because $\psi_f$, $\psi_g$ and $\psi_{fg}$ are homeomorphisms.

    \item We apply Theorem~\ref{thm:Proj_D=Proj^D} to \cite{paper1}, Lemma 1.30.
\end{enumerate}   
\end{proof}


\begin{remark}
    Lemma 3.4 in \cite{May} shows that the Brenner--Schröer construction coincides with the gluing construction, where affine opens are given by $\Spec(S_{(f)})$ for relevant $f$. 
    In particular, they prove that the canonical maps $S_{(fg)} \to S_{(f)}$ induce open immersions on spectra. 
    However, be aware that our notations for $D_+(f)$ differ.
\end{remark}


Also, the generalization of \cite{Bosch}, Lemma 9.1/14 holds true.

\begin{lemma}\label{lem:9.1/14_gen}
   Let $S$ be a factorially $D$-graded integral domain and $f, g \in \Gen^D_B(S)$ be elements of degree $d$, respectively $e$. 
    Then there is a canonical commutative diagram
    \begin{align*}
    \xymatrix{
      D_+(fg) \ar[rr]^{\psi_{fg}} \ar[dd]^{\iota} &   & \Spec(S_{(fg)}) \ar[dd]^{\sigma}  \\
      & & \\
      D_+(f)   \ar[rr]^{\psi_f} &    & \Spec(S_{(f)})         \, ,
      }
    \end{align*}
    where $\iota$ is the canonical inclusion and $\sigma$ is the open immersion obtained from the map of Lemma~\ref{lem:9.1/7_gen}:
    \begin{align*}
        S_{(f)} \to (S_{(f)})_\CG \stackrel{\sim}{\leftarrow} S_{(fg)}
    \end{align*}
\end{lemma}

\begin{proof}
    By \cite{Bosch}, Corollary 6.2/8 and as $\psi_f$ is a homeomorphism, we know that $\Spec(S_{(fg)}) = D_+(f) \cap D_+(g)$ and $\Spec(S_{(f)}) = D_+(f)$. In particular, $\sigma$ is an open immersion of schemes. The commutativity of the diagram follows from \cite{Bosch}, Lemma 9.1/14.
\end{proof}

\begin{definition}\label{def:B_vanishing} Let $(S, B)$ be a conical ring.
    For $Y \subseteq \Proj_B^D(S)$ we define the \emph{$B$-vanishing ideal} to be 
    \begin{align*}
        I_B(Y) \deq \{f \in B \mid Y \subseteq V_+(f)\} \equ I(Y) \cap B,
    \end{align*}
    where $I(Y) = \cap_{y\in Y} \Fp_y$, and for $f \in S$ we have $f \in \Fp_y$ if and only if $f(y) = 0$. For $B = S_+$ we denote $I_B(Y) = I_+(Y)$.
\end{definition}

The following statement is the multigraded version of the Nullstellensatz and the generalization of \cite{Bosch}, Corollary 9.1/15.
As far as we know, this version has not yet been stated.

\begin{corollary}[Multigraded Nullstellensatz]\label{cor:9.1/15_gen}
Let $(S, B)$ be a conical ring, such that $S$ is factorially $D$-graded.
\begin{enumerate}[label=(\alph*)]
    \item For any subset $E \subseteq B$ such that $(E)$ is a radical ideal, it holds $I_B(V_B(E)) = \rad_B(\Fa)$, where $\Fa$ is the restriction to $B$ of the graded ideal generated by $E$ in $S$ and $\rad_B(\Fa) = \rad(\Fa) \cap B$.
    \item For any subset $Y \subseteq \Proj^D_B(S)$, it holds: 
    \begin{align*}
        \overline{Y} \equ V_+(I_B(Y)) .
    \end{align*}
\end{enumerate}
\end{corollary}

\begin{proof}
\begin{enumerate}[label=(\alph*)]
    \item By Proposition~\ref{prop:hom_top} we have $V_+(E) = V_+(\rad(E)) = V_+(\Fa)$. As \cite{Bosch} indicates, it is enough to show that $\rad_+(\Fa)$ is the intersection of all $D$-prime ideals in $B$ containing $\Fa$. Thus, replacing $S$ with the quotient, we assume $\Fa = 0$. If some element $f \in S$ is not nilpotent, there exists at least one homogeneous component of $f$ that is not nilpotent. Hence, take some non-nilpotent $f \in S$ and assume without loss of generality that $f$ is homogeneous. Moreover, by replacing $f$ with $fg$ for some homogeneous element $g \in S$, we may assume that $f$ is relevant.
    As in \cite{Bosch}, it is enough to show that there exists a graded $D$-prime ideal $\Fp \unlhd S$ not containing $f$. But since $f$ is not nilpotent, $S_{(f)}$ is nonzero and hence contains a prime ideal $\Fq \in \Spec(S_{(f)})$. Then Theorem~\ref{thm:Proj_D=Proj^D} tells us that $\Fq$ corresponds to a graded $D$-prime ideal $\Fp \unlhd S$ such that $\Fp \in D_+(f)$ and hence $f \not\in \Fp$.

    \item Let $\Fa \subseteq B$ be a graded ideal of $S$ such that $Y \subseteq V_+(\Fa)$. Since $V_+(\Fa) = \bigcap_{f \in \Fa\cap \Rel^D_B(S)} V_+(f)$, $Y \subseteq V_+(f)$ for all $f$ and therefore $I_B(Y) \supset \Fa$. We conclude that $V_+(I_B(Y)) \subset V_+(\Fa)$ and hence $V_+(I_B(Y)) \subset \overline{Y}$. But clearly $Y \subset V_+(I_B(Y))$ and the claim follows.
\end{enumerate}
\end{proof}

\section{Separation Criteria}\makebox{}{}\label{sec:sep_crit}\\
There are many ways to think about the separability of $\Proj^D(S)$.
For example, in \cite{paper2} we showed that for semiprojective toric varieties $X_\Sigma$, $S = \Cox(X_\Sigma)$ and $D = \Cl(X_\Sigma)$, the scheme $\Proj^D(S)$ is exactly the direct limit of all GIT quotients of $X_\Sigma$. Hence, the global non-separatedness of Proj seems almost natural.

In this paper, we will present a straightforward algebraic approach to determine separated subsets, which comes with a technical flavor.
We want to emphasize, that in fact all that is required for separability is the surjectivity of all the multiplication maps $\mu_{(fg)} \colon S_{(f)} \otimes S_{(g)} \to S_{(fg)}$ for $f, g \in \Rel^D(S)$ (cf.\  \cite{BS}, proof of Proposition 3.3 and \cite{Bosch}, §9.1).
For more details on this subject, we refer to the author's thesis\footnote{\url{https://sites.google.com/view/felixgoebler/home/research}}.



\begin{example}\label{ex:double_origin_sep}
    In the situation of Example~\ref{ex:double_origin_C_grading}, it holds that $S_{(yz)} = S_{(xz)} = S_{(z)} = \BC[\frac{xy}{z}]$. Since $S_{(xz\cdot yz)} = \BC[\frac{xy}{z}, \frac{z}{xy}]$, $\mu_{(xz\cdot yz)}$ is not surjective.
    As $\Proj^D(S)$ is covered by the basic affine opens $D_+(f)$ for $f = xy,xz,yz$, it cannot be separated.
Note that the non-separatedness mainly follows from the left-hand side of the linear dependency $\deg(xy) = \deg(z)$ in $D_\BR$. In fact, we expect $\Proj^D(S)$ to be non-separated as soon as there is a linear dependency of that type.
\end{example}

Therefore, we want to formalize Example~\ref{ex:double_origin_sep}.

\begin{definition}
Let $S$ be a factorially $D$-graded ring.
\begin{enumerate}[label=(\arabic*)]\label{def:weak}
    \item A homogeneous element $s \in S$ is called \emph{$D$-invertible}, if there exists a coprime homogeneous $r \in S$ such that $(r, s)$ is a non-trivial element of degree zero in $\Quot(S)$.

        \item  Let $f, g \in \Gen^D_B$. We say that $(f, g)$ is a \emph{weak pair} (or only weak), if there is a $D$-invertible element $h$ satisfying $h \mid fg$ but $h \not\mid f$ and $h \not\mid g$, such that there exists an element $x \in S$ such that $(x, h)$ is a generator of degree zero in $\Quot(S)$.  

    \item We say that $f \in \Gen^D_B$ is \emph{weak}, if there is a relevant element $g \in \Gen^D_B$ such that $(f, g)$ is weak.  
\end{enumerate}

    Note that each $D$-invertible element $s \in S$ corresponds to a linear dependency of type $\deg(s) = \deg(f \cdot g)$ for some $f, g \in S$ (where possibly $f=1$ or $g = 1$). We will call an irreducible proper homogeneous divisor of a generator of degree zero (i.e.\ it divides the numerator or the denumerator) a \emph{linear dependent variable of $S$}. 
\end{definition}

The strength of this language lies in the following fact

\begin{proposition}\label{prop:surj_cond}
Let $S$ be factorially graded.
The relevant pair $(f, g)$ for $f, g \in \Gen^D$ is weak if and only if $\mu_{(fg)}$ is not surjective.
\end{proposition}

\begin{proof}
\Ra Clearly, $h$ corresponds to an $D$-invertible element in $S_{(fg)}$, i.e.\ there exists an element $x \in S$ such that $(x, h)$ is a generator of degree zero in $S_{(fg)}$. In particular, $\frac{x}{h}$ cannot be written as a proper product (of non-trivial elements of degree zero) and thus cannot have a preimage.\\

\La Let $\mu_{(fg)}$ be not surjective, i.e.\ there is some generator of degree zero $(x', h') \in S_{(fg)}$ that has no preimage. In particular, $h'|fg$, but $h'$ cannot divide neither $f$ nor $g$. 
\end{proof}

\begin{example}\label{ex:ush_sep1}
We compute the weak elements for several rings.
\begin{enumerate}[label=(\arabic*)]
    \item Let $S = \BC[x, y, z]$ and $D = \BZ^2$, where $\deg(x) = e_1$, $\deg(y) = e_2$ and $\deg(z) = e_1+e_2$. Then $f = xz$ is not $D$-invertible, as the only non-trivial element of degree zero with $f$ as denominator is $\frac{x^2y}{xz} = \frac{xy}{z}$, hence $f$ cannot occur as denominator in reduced fractional form. The same holds true for $g = yz$. But the relevant element $h = xy$ is $D$-invertible and $(z, xy) \in S_{(fg)}$ is a generator of degree zero that cannot have a preimage (with respect to $\mu_{(fg)}$). Also note that
     $x$ as well as $y$ are linear dependent variables of $S$, that are not $D$-invertible (i.e.\ there are non $D$-invertible divisors of a $D$-invertible element).
     In particular, $(xz, yz)$ is the only weak pair belonging to $S$.

    \item Let $S = \BC[x, y, z]$ be graded by $\BZ\times \BZ/2\BZ$, such that $\deg(x) = (1, 0)$, $\deg(y) = (0, 1)$ and $\deg(z) = (1, 1)$. In this case, $\Gen^D(S) = \{x, z\}$.
    In contrast to (2), monomic relevant elements have length one, so $h = xy$ cannot be split. In particular, $S$ does not contain a weak element and $\Proj^D(S)$ is separated.
\end{enumerate} 
\end{example}

We found several conditions implying the existence of a weak pair:

\begin{lemma}\label{lem:weak_cond} 
Let $S$ be factorially graded, $B \subseteq S_+$ a subset of the irrelevant ideal.
\begin{enumerate}[label=(\alph*)]
    \item Let $f, g \in \Gen^D_B(S)$ such that there is a $D$-invertible element $h$ 
with at least one non-$D$-invertible divisor (say $d_h$), such that $h \mid fg$ but $h \not\mid f$ and $h \not\mid g$. Then $(f, g)$ is weak.

    \item Let $f \in \Gen^D_B(S)$ be a product of linear dependent variables of $S$, such that $f$ is not $D$-invertible. Then $f$ is weak.

    \item Let $f\in \Gen_B^D(S)$ and $f' \in S$ be a proper homogeneous divisor of $f$, such that $S_{(f')} = S_{(f)}$ and $f = f'g$, where $g$ is a product of linear dependent variables. Then $f$ is weak.
\end{enumerate} 
\end{lemma}

\begin{proof}
\begin{enumerate}[label=(\alph*)]
    \item Since $h \mid fg$ is $D$-invertible, clearly there exists an element $x \in S$ such that $(x, h)$ is a non-trivial element of degree zero (in $S_{(fg)}$). Since $h$ is neither $D$-invertible in $S_{(f)}$ nor in $S_{(g)}$ and $d_h$ is not $D$-invertible at all, $(x, h)$ has to be irreducible, i.e.\ a generator of degree zero.
    In particular, there are no $y, h' \in S$, $y \mid x$ and $h'\mid h$, such that $(y, h')$ is a non-trivial element of degree zero.

    \item Since $f$ is not $D$-invertible, it has a non $D$-invertible divisor, say $d_f$. Let $d_f \neq f' \mid f$ be an irreducible divisor (without loss of generality, we may assume $r=2$).
    Since $d_f$ is a linear dependent variable, there exists $d_g$ such that $d_f d_g$ is $D$-invertible. In particular, $(f, f'd_g)$ is weak.

    \item By (b), we know that $g$ satisfies the condition of a weak element, but $g$ is not relevant. In particular, $f$ is weak.
\end{enumerate}
\end{proof}

\begin{example}
\begin{enumerate}[label=(\arabic*)]
    \item  Let $S = \BC[x, y, z, w]$ and $D = \BZ^2$, where $\deg(x) = \deg(y) = (1, 0)$, $\deg(z) = (1, 1)$ and $\deg(w) = (0, 1)$. It holds that $\Gen^D(S) = \{xw, yw, zw, xz, yz\}$. Then the conditions from Lemma~\ref{lem:weak_cond} (a) are satisfied for $f = wz$ and $g = xz$ (or $yz$), as $h := xw$ (resp. $h = yw$) has the non-$D$-invertible divisor $w$.
    \item Let $S = \BC[x, y, z, v, w]$ where the grading is given by $x, y \mapsto e_1$, $v, w \mapsto e_2$ and $z \mapsto e_1 + e_2$, so $S_+ = (xz, xw, xv, yz, yw, yv, zw, zv)$. Note that every variable is $D$-invertible, even though we have non-trivial linear dependencies. But still, $(xz, wz)$, $(xz, vz)$, $(yz, wz)$ and $(yz, vz)$ are weak pairs.
\end{enumerate}
\end{example}


Next, we connect the notion of weak pairs to the linear dependency types in $D$. If we assume $S$ to be factorially graded, we can talk about the \emph{length} of a homogeneous element in terms of the number of its $D$-prime factors. We have to be very careful due to torsion.

\begin{definition}
Let $S$ be factorially graded, and $h, s \in S$ be homogeneous. We say that $\deg(h)$ and $\deg(s)$ are \emph{non-trivial irreducible dependent}, if $\deg(h) = \deg(s)$ and 
\begin{enumerate}[label=(\roman*)]
    \item $\length(h, s) \neq (1, 1)$ and
    \item $\nexists h', s'\in S, k,l\in \BN: h'\mid h^k, s'\mid s^l$ such that $(h', s')$ is a non-trivial element of degree zero, where $h \neq h'$ and $s' \neq s$.
\end{enumerate}
If only (ii) holds, we say that $\deg(h)$ and $\deg(s)$ are \emph{irreducible dependent}.
Now let $f \in S_d$ be relevant.  
We call a homogeneous element $a$ satifying $a \mid f$ an \emph{irreducible component of $f$}, if $a$ is part of an irreducible dependency, i.e.\ if there are homogeneous elements $b, c$ such that $\deg(ab) = \deg(c)$ or $\deg(a) = \deg(bc)$.
Likewise, we call all the irreducible homogeneous elements belonging to the linear dependency $\deg(ab) = \deg(c)$ and not dividing $f$ \emph{irreducible components not dividing $f$}. 
\end{definition}

\begin{example}
    In Example~\ref{ex:ush_sep1} we consider the equation $\deg(xy) = \deg(z)$. While in (1), there is no other equation in $D_\BR$ and thus they are non-trivial irreducible dependent, the situation in (2) is completely different: We see that for $h' = x^2$ and $s' = z^2$, the equation $\deg(xy) = \deg(z)$ becomes dependent after squaring it. In particular, $\Proj^D(S)$ is separated in (2).
\end{example}

\begin{proposition}\label{prop:weak_cond} 
Let $S$ be factorially graded, and $(h, s)$ be non-trivial irreducible dependent, where we may assume $\length(h) > 1$ without loss of generality. Then there exists a weak pair $(f, g)$ such that $h \mid fg$ but $h\nmid f$, $h \nmid g$.
\end{proposition}

\begin{proof}
Take homogeneous elements $f', g' \in S$ of length $d-\length(s)$ such that $h \mid f'g'$ and $h \not\mid f',g'$.
This is possible since $\length(h) > 1$, i.e.\ the divisors of $h$ don't give rise to another non-trivial element of degree zero (in $\Quot(S)$).
Then, without loss of generality, $(f, g) = (f's, g's)$ is a weak pair with the desired properties. 
\end{proof}

With this at hand, we can deduce a separation criterion in terms of linear dependencies in the factorially graded setting.

\begin{theorem}\label{thm:lin_dep_not_sep}
Let $S$ be a factorially $D$-graded integral domain. Then the following holds true:
\begin{enumerate}
    \item If there are only linear dependencies of length $1$, $\Proj^D(S)$ is separated.

    \item If there exists at least one non-trivial irreducible dependency, then $\Proj^D(S)$ is not separated.
\end{enumerate}
\end{theorem}

\begin{proof}
    \begin{enumerate}
        \item We already know that $\mu_{(fg)}$ is not surjective if and only if $(f, g)$ is weak. Since every linear dependency is of type $\deg(x_i) =  \lambda \deg(x_j)$ for linear dependent variables $x, y \in S$, there is no $h\mid fg$ such that $h$ (resp. the divisors of $h$) is not invertible in $S_{(f)}$ or $S_{(g)}$.

        \item Follows from Proposition~\ref{prop:weak_cond}.
    \end{enumerate}
\end{proof}

If we view $S$ as conical ring, we get separated models by choosing \emph{nice} ideals $B$:

\begin{corollary}[Separability Criteria for conical rings]\label{cor:sep_crit}\makebox{}{}\\
For a factorially graded conical ring $(S, B)$, the scheme $\Proj^D_B(S)$ is separated if and only if $B$ does not contain a weak element. 
\end{corollary}

\begin{proof}
We only have to show that $\mu_{(fg)}$ is surjective. This immediately follows from Proposition~\ref{prop:surj_cond}.
\end{proof}

\begin{example}
    We want to consider Example~\ref{ex:ush_sep1} (1) again. We already know that $(xz, yz)$ is the only weak pair of $S$. Hence for $B_1 = (xz, xy)$ and $B_2 = (yz, xy)$ it holds
    \begin{align*}
        \Proj^D_{B_1}(S) \equ \BP_\BC^1 \equ \Proj^D_{B_2}(S).
    \end{align*}
    This coincides with the decomposition $D_+(x) = D_+(xz) \cup D_+(xy) \cong \BP^1_\BC \cong D_+(yz) \cup D_+(xy) = D_+(y)$ from Corollary~\ref{cor:intersection_and_open_subsets_of_hom_elements} (c).
\end{example}

\section{Sheaves on Multigraded Spectra}\label{subsec:sheaves}

Sheaves on multigraded spectra have already been studied by \cite{May} in chapter 5. Hence, we will not give too many details here.

For a $D$-graded $S$-module $M$ and any $d \in D$, we define the graded $S$-module $M(d)$ by $M(d)_e \equ M_{e+d}$ for $e \in D$.
For $f \in S_d$, $d\neq 0$, $M_f$ may be viewed as a graded $S_f$-module, and its homogeneous part of degree zero $M_{(f)}$ is an $S_{(f)}$-module, called the \emph{homogeneous localization of $M$ by $f$}. 

By \cite{May}, Fact 5.1 and Fact 5.2, we get the following generalization of \cite{Bosch}, Proposition 9.2/1. Note that in contrary to \cite{May}, our notation for $D_+(f)$ denotes the set of $D$-prime ideals not containing $f$, which coincides with $\Spec(S_{(f)})$ by Theorem~\ref{thm:Proj_D=Proj^D}.
For the rest of this paper, we denote $X = \Proj^D(S)$.

\begin{proposition}\label{prop:glue_M_f_qc}
Let $S$ be an integral factorially $D$-graded domain, $X = \Proj^D(S)$. Then the $\CO_{D_+(f)}$-modules $\widetilde{M}_{(f)}$, where $f$ varies over all relevant elements of $S$, can be glued to a unique quasi-coherent $\CO_X$-module $\widetilde{M}$, so that it holds
\begin{align*}
    \Gamma(D_+(f), \widetilde{M}) \equ M_{(f)}.
\end{align*}
Furthermore, the assignment $M \mapsto \widetilde{M}$ is functorial, mapping from the (abelian) category of $D$-graded $S$-modules to the (abelian) category of quasi-coherent sheaves on $\Proj^D(S)$. 
In particular, the functor is exact and factors through $\textrm{Mod}^D(S)/\textrm{Mod}^D(S)_{\textrm{neg}} \to \textrm{QCoh}(\Proj^D(S)$, where $M \in \textrm{Mod}^D(S)_{\textrm{neg}}$ if and only if $\widetilde{M}=0$.
\end{proposition}

The next objective would be to check if $\stackrel{\resizebox{11mm}{1mm}{$\sim$}}{M \otimes_S N} \cong \widetilde{M} \otimes_{\CO_X} \widetilde{N}$, like in \cite{Bosch}, Proposition 9.1/2. However, there is a condition that $S_+$ is generated by $S_1$, which will not be the case in general, as we have seen in the examples. Hence, it is no surprise that we only get a canonical map in general:

\begin{proposition}\label{prop:module_sheaf_tensor}
    Let $M, N$ be $D$-graded $S$-modules. Then there is a canonical map of $\CO_X$-modules
    \begin{align*}
       \widetilde{M} \otimes_{\CO_X} \widetilde{N} \ \to \  \stackrel{\resizebox{11mm}{1mm}{$\sim$}}{M \otimes_S N},
    \end{align*}
    inducing a canonical map
    \begin{align*}
        M_{(f)} \otimes_{S_{(f)}} N_{(f)} \ \to \ (M\otimes_S N)_{(f)}
    \end{align*}
    for relevant $f \in S$.
\end{proposition}

\begin{proof}
    \cite{May}, Proposition 5.3.
\end{proof}

We are mainly interested in generalizing twisting sheaves, as we want to construct ample families with them.
The following definition is due to \cite{May}.

\begin{definition}\label{def:twisting_sheaves} Let $d \in D$, $X = \Proj^D(S)$.
\begin{enumerate}[label=(\arabic*)]
    \item The quasi-coherent sheaf $\stackrel{\resizebox{11mm}{1mm}{$\sim$}}{S(d)}$ is called the \emph{$d$-th twist of the structure sheaf} of $\Proj^D(S)$ and is denoted by $\CO_X(d)$.
    \item For arbitrary $\CO_X$-modules $\CM$ we set
    \begin{align*}
        \CM(d) \deq \CO_X(d) \otimes_{\CO_X} \CM .
    \end{align*}
\end{enumerate}    
\end{definition}

In the noetherian case, \cite{May} showed that $\CO_X(d)$ is coherent:

\begin{lemma}\label{lem:O_X_coherent}
    Let $S$ be noetherian.
    \begin{enumerate}[label=(\alph*)]
        \item For all $d \in D$, $\CO_X(d)$ is coherent.
        \item If $M$ is a finitely generated $D$-graded $S$-module, then $\widetilde{M}$ is coherent.
    \end{enumerate}
\end{lemma}

\begin{proof}
    \cite{May}, Lemma 5.6.
\end{proof}

For $d, e \in D$, it holds $S(d) \otimes S(e) = S(d+e)$. This induces a canonical map
\begin{align}\label{eq:O(d)_O(e)_to_O(d+e)}
 \phi_{d, e}\colon    \CO_X(d) \otimes_{\CO_X} \CO_X(e)\ \to\ \CO_X(d+e) 
\end{align}
by Proposition~\ref{prop:glue_M_f_qc}.

Note that there are canonical maps
\begin{align*}
   \phi_S\colon S \ \to \ \bigoplus_{d\in D} \Gamma(X, \CO_X(d))
\end{align*}
and
\begin{align*}
    M \ \to \ \bigoplus_{d\in D} \Gamma(X, \stackrel{\resizebox{11mm}{1mm}{$\sim$}}{M(d))}
\end{align*}
by \cite{May}, Fact 5.7, giving $\bigoplus_{d \in D} \CM(d)$ the structure of an $D$-graded sheaf of modules on  $\bigoplus_{d\in D}\CO_X(d)$ (for any quasi-coherent $\CO_X$-module $\CM$), and $\bigoplus_{d\in D} \Gamma(X, \CM(d))$ the structure of an $D$-graded $\bigoplus_{d\in D} \Gamma(X, \CO_X(d))$-module.
Thus we get a functor
\begin{align*}
    \Gamma_\bullet\colon \mathrm{Mod}(\CO_X) \to \mathrm{Mod}^D(S).
\end{align*}

\begin{lemma}\label{lem:functor_OX_mod_to_mod}
    For each quasi-coherent $\CO_X$-module $\CM$, there exists a functorial morphism 
    \begin{align*}
        \stackrel{\resizebox{11mm}{1mm}{$\sim$}}{\Gamma_\bullet(\CM)} \to \CM.
    \end{align*}
    in $\mathrm{Mod}(\CO_X)$.
\end{lemma}

\begin{proof}
    Follows from \cite{May}, Fact 5.7(i).
\end{proof}

From this point, we present our own results.

For an element $s = \sum_{d\in D}s_d \in S$, we define $\sigma_{s, d} \in \Gamma(X, \CO_X(d))$ to be the section obtained from gluing together all sections $\frac{s_d^l h_s}{f^m}$, such that $(s_d, f)$ admits an element of degree $d$ (in $S_{(f)}$), where $f$ varies over $\Gen^D(S)$. Hence, $\varphi_S(s) = \sum_{d\in D}\sigma_{s,d }$. In particular, we have an induced morphism of rings $S_d\to \Gamma(X, \CO_X(d))$, given by $s_d \mapsto \sigma_{s, d}$.

\begin{lemma}\label{lem:section_S_d_nice}
Let $S$ be a factorially $D$-graded noetherian integral domain, and let $d \in \supp(D)$.
    Then the induced map
    \begin{align*}
      \phi_{S, d}\colon  S_d \to \Gamma(X, \CO_X(d))
    \end{align*}
    is an isomorphism of rings.
\end{lemma}

\begin{proof}
        Let $s\in S_d$ be nonzero and $f \in \Gen^D(S)$, so that $(s, f)$ admits an element of degree $d$. As we allow $(s, f)$ to be $\frac{s}{1}$, such an $f$ exists by Corollary~\ref{cor:rel_many_el_deg_zero}. Therefore the image of $s$ in $\Gamma(D_+(f), \CO_X(d))$ in nonzero and hence also the image of $s$ in $\Gamma(X, \CO_X(d))$. So let $\varphi_{S, d} = 0$. Then $s\restrict_{D_+(f)} = 0$ in $S_{(f)}(d)$ for all $f \in \Gen^D(S)$. As the $D_+(f)$ cover $\Proj^D(S)$ and $S$ is integral, $s = 0$ in $S_f$ for all $f \in \Gen^D(S)$ forces $s = 0$ and $\varphi_{S, d}$ is injective.

        In order to proof surjectivity, take $\sigma \in \Gamma(X, \CO_X(d))$. Then $\sigma|_{D_+(f)} \in \Gamma(D_+(f), \CO_X(d))$ for all $f \in \Gen^D(S)$. Since $\Gamma(D_+(f), \CO_X(d)) \cong S(d)_{(f)}$ (cf.\  \cite{May}, Proposition 5.8), there exists an element $s_f \in S_{d+l_f \deg(f)}$ for all $f \in \Gen^D(S)$, such that $\sigma_f=\frac{s_f}{f^{l_f}}$ has degree $d$. 
        Since these representations have to be compatible on $D_+(fg)$ for all $f, g \in \Gen^D(S)$ and $S$ is factorially graded, we may assume that these local sections glue to the section $\sigma = \frac{a}{b}$, where $b$ divides some relevant element, $a \in S_{\deg(b) + d}$ and $\sigma$ is in reduced fractional form. If $b$ is not a unit in $S$, we may factorize it into $D$-prime factors, say $b = gh$. Since $\sigma$ is a nonzero section, the restriction to $D_+(g)$ is also nonzero. By factoriality, we may assume that the numerator of the restriction of $D_+(g)$ multiplied by $h$ equals $a$, hence $h \mid a$. This contradicts the reduced fractional form, so $b$ was a unit to begin with and $\deg(\sigma) = \deg(ab^{-1}) = d$, i.e.\ $\sigma = \phi_{S,d}(ab^{-1})$.
\end{proof}

\subsection{Freeness and reflexivity of Serre twists}

If we take $S$ to be arbitrary and allow torsion in $D$, we are facing two big differences (cf.\  \cite{Cox}, page 12):
\begin{enumerate}[label=(\arabic*)]
    \item $\CO_X(d)$ may not be a line bundle.
    \item $\phi_{d,e}$ might not be an isomorphism.
\end{enumerate}

However, we can detect if $\CO_X(d)$ is a line bundle surprisingly simply, at least for toric $X$. Since $X$ is then $\BQ$-factorial by \cite{paper4}, Corollary 2.33. In particular, all $\CO_X(d)$ are invertible.

\begin{proposition}\label{prop:O_X(d)_inv_free_mod}
    Let $X = \Proj^D(S)$, $d \in D$. Then the following are equivalent:
    \begin{enumerate}[label=(\alph*)]
        \item $\CO_X(d)$ is free as an $\CO_X$-module.
        \item For all relevant $f \in S$, $S(d)_{(f)}$ is a free $S_{(f)}$-module of rank one and $S_f$ contains a unit $u_f$ in degree $d$. 
        \item For all relevant $f \in S$, $S_f$ contains a unit $u_d$ in degree $d$ such that multiplication with $u_d$ defines an isomorphism $S_{(f)} \to S(d)_{(f)}$ of $S_{(f)}$-algebras.
    \end{enumerate}
    If $S$ is noetherian and factorially graded, (b) and (c) simplify, so that the conditions only have to be checked for all $f \in \Gen^D(S)$.
\end{proposition}

\begin{proof}
$(a) \implies (b):$ Let $\CO_X(d)$ be free. As $\Proj^D(S)$ is covered by the set of all $\Spec(S_{(f)})$ for relevant $f \in S$, the claim follows by the definition of a line bundle.

$(b) \implies (a):$ We clearly have the following trivialization data for $L = \CO_X(d)$: Take $U_f = \Spec(S_{(f)})$ and $\psi_f$ to be the isomorphism induced by the isomorphism $S_{(f)} \to S(d)_{(f)}$ that is given my multiplication with $u_f$. Then by construction, $\frac{u_f}{u_g}$ is a unit in $S_{(fg)}$ and $\CO_X(d)$ is a line bundle.
The equivalence of $(b)$ and $(c)$ is trivial, as $(b)$ means that $S(d)_{(f)}$ is a free $S_{(f)}$-module of rank $1$ with basis $u_f$.
\end{proof}

The following corollary directly follows from \cite{paper1}, Theorem 2.7.

\begin{corollary}\label{thm:O(d)_free}
    Let $S$ be an integral factorially graded domain, $d\in D$ and $(f_i)_{i\in I}$ a system of generators for $S_+$. 
    Then $\CO_X(d)$ is free if and only if $d \in \bigcap_{i \in I} D^{f_i}$.
\end{corollary}

\begin{proof}
    \Ra By Proposition~\ref{prop:O_X(d)_inv_free_mod} and factoriality, we know that $S(d)_{(f_i)}$ contains a unit $u_{f_i}$ in degree $d$, which is equivalent to $d \in \bigcap_{i \in I} D^{f_i}$.

    \La Conversely, let $d$ be contained in the intersection of all $D^{f_i}$. Then there must exist a unit $u_i \in S_{f_i}$ of degree $d$ for all $i \in I$ and thus, the map $S_{(f_i)} \to S(d)_{(f_i)}$ given by $q \mapsto u_i q$ is an isomorphism and turns $S(d)_{(f_i)}$ into a free $S_{(f_i)}$-module of rank 1.
\end{proof}

\begin{example}\label{ex:invertible}
We consider the graded rings from Example~\ref{ex:ush_sep1} again.
    In (1), every relevant element is strongly relevant, so $D^f = D$ and hence, all $\CO_X(d)$ for $d \in D$ are invertible.

    For (2) we compute that $D^x \cap D^z = \BZ \cap (\BZ \times \BZ/2\BZ) = 2 \BZ \times \{0\}$, which shows that $\CO_X(d)$ is free if and only if $d = k (2,0)$ for $k \in \BZ$. 
    However, as $\CO_X(0,1)^{\otimes 2} = \CO_X(0, 0) = \CO_X$, it is also invertible. Hence, $\CO_X(1,1)$ and $\CO_X(1,0)$ are also invertible. In particular, choosing $O_X(1,0)$ and $\CO_X(1,1)$ as members of the ample family, we recover the relevant locus $S_+ = (x, z)$ and hence its Proj. Note, again, that this choice is not canonical.
\end{example}

In the remaining section, we will determine when $\CO_X(d)$ is torsion-free and reflexive.

\begin{definition}\label{def:reflexive_sheaf}
    Let $X$ be a scheme and $\CF$ a coherent sheaf on $X$ and consider the dual sheaf $\CF^\vee = \CH\text{om}_{\CO_X}(\CF, \CO_X)$. Let $\CG$ be an $\CO_X$-module.
    \begin{enumerate}[label=(\arabic*)]
        \item We say that $\CF$ is \emph{reflexive}, if the natural map $\CF \to (\CF^\vee)^\vee$ is an isomorphism.
        \item $\CG$ is called \emph{torsion-free}, if for all $U \subseteq X$ and all $0 \neq r \in \CO_X(U)$ and all $0 \neq z \in \CG(U)$ it holds $rz \neq 0$.
    \end{enumerate}    
\end{definition}

\begin{proposition}
Let $S$ be an integral $D$-graded domain, $d \in D$. Then
    \begin{enumerate}[label=(\alph*)]

        \item $\CO_X(d)$ is torsion-free of generic rank $1$ 
    \end{enumerate}
    Let $S$ also be noetherian. Then
    \begin{enumerate}

        \item[(b)] $\CO_X(d)^\vee$ and $(\CO_X(d)^\vee)^\vee$ are torsion-free. 

        \item[(c)] the natural map 
    \begin{align*}
        \alpha_d\colon \CO_X(d) \to (\CO_X(d)^\vee)^\vee
    \end{align*}
    is injective.

        \item[(d)] $\CO_X(d)^\vee$ and $(\CO_X(d)^\vee)^\vee$ are reflexive.
    \end{enumerate}
\end{proposition}

\begin{proof}
    \begin{enumerate}[label=(\alph*)]
        \item We observe that $\CO_X(d)$ has torsion if and only if there exists some relevant $f \in S$ and non-zero elements $r \in S_{(f)}$, $z \in S(d)_{(f)}$ such that $rz = 0$. But this multiplication lives inside $\Quot(S)$, so there cannot exist such elements $r, z$ for no relevant $f \in S$. Hence $\CO_X(d)$ is torsion-free.
    Let $\eta = (0)$ be the generic point of $X$. Then 
    \begin{align*}
        \CO_X(d) \otimes \kappa(\eta) \equ S(d) \otimes_S \Quot(S) \cong \Quot(S),
    \end{align*}
    so $\CO_X(d)$ has fiber of dimension one over $\Quot(S)$.

    \item As $\CO_X(d)$ is coherent by Lemma~\ref{lem:O_X_coherent} and $\Proj^D(S)$ is integral (by \cite{KU}, Lemma 3.6) and noetherian, the claim follows by \cite{Schwede}, Lemma 2.4.

    \item As $\CO_X(d)$ is torsion-free by (a), we can apply \cite{Schwede}, Lemma 2.5.    

    \item Again, since $\CO_X(d)$ is coherent and $X$ is noetherian and integral, we can apply \cite{Schwede}, Theorem 2.8.
    \end{enumerate}
\end{proof}

\begin{remark}
    \begin{enumerate}
        \item[(1)] Note that by \cite{Har94}, Theorem 1.9 (also see \cite{Schwede}, Theorem 2.10), $\CO_X(d)$ is reflexive if and only if it is $S_2$, for $S$ being an integral $D$-graded noetherian and normal domain.
        \item[(2)] It is well-known that any finite locally free $\CO_X$-module is reflexive. In particular, this applies for $S$ being an integral $D$-graded noetherian domain such that $\CO_X(d)$ is finite locally free, so that $\CO_X(d)$ is also reflexive.
    \end{enumerate}
\end{remark}

\end{document}